\pdfoutput=1
\RequirePackage{ifpdf}
\ifpdf 
\documentclass[pdftex]{sigma}
\else
\documentclass{sigma}
\fi

\newcommand{\Ric}{\mathrm{Ric}}

\newcommand{\e}{\epsilon}

\renewcommand{\S}{\Sigma}

\renewcommand{\Im}{\mathrm{Im} \,}

\newcommand{\E}{\mathcal E}
\newcommand{\A}{\mathcal A}

\newcommand{\supp}{\mathrm{supp}}

\newcommand{\de}{\delta}
\renewcommand{\b}{\beta}
\renewcommand{\d}{\partial}

\newcommand{\abs}[1]{\left\lvert#1\right\rvert}

\renewcommand{\H}{\mathrm H}
\newcommand{\Ho}{\mathcal H}
\renewcommand{\E}{\mathcal E}

\newcommand{\U}{\mathcal U}

\newcommand{\T}{\mathrm T}
\newcommand{\md}{\mathrm d}
\newcommand{\p}{\mathrm p}
\newcommand{\q}{\mathrm q}

\newcommand{\Char}{\mathrm{Char}}
\newcommand{\s}{\sigma}

\newcommand{\sgn}{\operatorname{sgn}}

\newcommand{\R}[0]{\mathbb{R}}
\newcommand{\C}[0]{\mathbb{C}}

\numberwithin{equation}{section}

\newtheorem{Theorem}{Theorem}[section]
\newtheorem*{Theorem*}{Theorem}
\newtheorem{Corollary}[Theorem]{Corollary}

\newtheorem{Proposition}[Theorem]{Proposition}
 { \theoremstyle{definition}
\newtheorem{Definition}[Theorem]{Definition}

\newtheorem{Remark}[Theorem]{Remark} }
\newtheorem{Assumption}[Theorem]{Assumption}

\begin{document}


\renewcommand{\thefootnote}{}

\newcommand{\arXivNumber}{2306.09213}

\renewcommand{\PaperNumber}{052}

\FirstPageHeading

\ShortArticleName{Stationarity and Fredholm Theory in Subextremal Kerr--de Sitter Spacetimes}

\ArticleName{Stationarity and Fredholm Theory\\ in Subextremal Kerr--de Sitter Spacetimes\footnote{This paper is a~contribution to the Special Issue on Global Analysis on Manifolds in honor of Christian B\"ar for his 60th birthday. The~full collection is available at \href{https://www.emis.de/journals/SIGMA/Baer.html}{https://www.emis.de/journals/SIGMA/Baer.html}}}

\Author{Oliver PETERSEN~$^{\rm a}$ and Andr\'as VASY~$^{\rm b}$}

\AuthorNameForHeading{O.~Petersen and A.~Vasy}

\Address{$^{\rm a)}$~Department of Mathematics, Stockholm University, 10691 Stockholm, Sweden}
\EmailD{\href{mailto:oliver.petersen@math.su.se}{oliver.petersen@math.su.se}}

\Address{$^{\rm b)}$~Department of Mathematics, Stanford University, CA 94305-2125, USA}
\EmailD{\href{mailto:andras@math.stanford.edu}{andras@math.stanford.edu}}

\ArticleDates{Received June 16, 2023, in final form June 10, 2024; Published online June 20, 2024}

\Abstract{In a recent paper, we proved that solutions to linear wave equations in a~subextremal Kerr--de~Sitter spacetime have asymptotic expansions in quasinormal modes up to a~decay order given by the normally hyperbolic trapping, extending the results of Vasy~(2013). One central ingredient in the argument was a new definition of quasinormal modes, where a non-standard choice of stationary Killing vector field had to be used in order for the Fredholm theory to be applicable. In this paper, we show that there is in fact a variety of allowed choices of stationary Killing vector fields. In particular, the horizon Killing vector fields work for the analysis, in which case one of the corresponding ergoregions is completely removed.}


\Classification{35L05; 35P25; 58J45; 83C30}


\renewcommand{\thefootnote}{\arabic{footnote}}
\setcounter{footnote}{0}

\section{Introduction}

Kerr--de~Sitter spacetimes model stationary rotating black holes in an otherwise empty expanding universe.
They solve Einstein's vacuum equation $\Ric(g) = \Lambda g$ with a positive cosmological constant $\Lambda > 0$.
The mass of the black hole is denoted $m > 0$ and the angular momentum of the black hole is denoted $a \in \R$.
The Kerr--de~Sitter black hole admits two horizons, called the event horizon, located at radius $r_{\rm e}$, and the cosmological horizon, located at radius $r_{\rm c}$.
If a future directed causal curve is ever at a~radius smaller than $r_{\rm e}$ (resp.~larger than $r_{\rm c}$), it will never be able to reach a radius larger than $r_{\rm e}$ (resp.~smaller than $r_{\rm c}$), meaning that it is stuck forever beyond the event horizon (resp.~beyond the cosmological horizon).
Thus, the cosmological horizon has very similar role as the event horizon (which is the boundary of the black hole) both from a~physics and a mathematics perspective.
The radii $r_{\rm e}$ and $r_{\rm c}$ are the two largest roots of a~certain quartic polynomial $\mu$, given in terms of $\Lambda$, $a$ and $m$ in equation \eqref{eq: mu} below.
We say that the Kerr--de~Sitter black hole is \emph{subextremal} if $\mu$ has four distinct real roots.

Just like the closely related Kerr spacetime, which can be seen as the limit when $\Lambda \to 0$, the Kerr--de~Sitter spacetime has a rotational symmetry, given in the standard Boyer--Lindquist coordinates (see Definition~\ref{def: BL metric}) by the Killing vector field $\d_\phi$, and a stationarity symmetry, given in the same coordinates by for example the Killing vector field $\d_t$.
However, there is an ambiguity in what precise symmetry should describe the stationarity of the black hole.
Indeed,
$c_1 \d_t + c_2 \d_\phi$
is a Killing vector field, for any constants $c_1, c_2 \in \R$.
Moreover, one can check that no such Killing vector field is timelike everywhere between the horizons (cf.~Remark~\ref{rmk: time inv}).
In the Kerr spacetime, there is on the other hand a canonical choice of Killing vector field to describe the stationarity, $\d_t$ namely is the only one which is timelike at large distances from the black hole.
In the Kerr--de~Sitter spacetime, there is no such analogue, and it is a priori not clear what Killing vector field should be modeling the stationarity.

The purpose of this paper is to illustrate that, in the Kerr--de~Sitter spacetime, many natural properties are satisfied if we choose
\[
	\T
		:= \d_t + \frac a{r_0^2 + a^2}\d_\phi
\]
to be the stationary Killing vector field, where we may choose any $r_0 \in [r_{\rm e}, r_{\rm c}]$,
where $r_{\rm e}$ is the radius to the event horizon and $r_{\rm c}$ is the radius to the cosmological horizon.
Note that as~${\Lambda \to 0}$, then $r_{\rm c} \to \infty$, so in the limit, an allowed choice for~$\T$ is indeed the standard choice~$\d_t$ in the Kerr spacetime.
The main new observation in this paper is that there are no trapped lightlike geodesics with trajectories orthogonal to $\T$.
In fact, the geodesic flow of the lightlike geodesics with trajectories orthogonal to $\T$ is very similar to the much easier case when~${a = 0}$ (where indeed $\T = \d_t$).
However, it was observed in \cite[p.~486]{V2013} that this is not the case if we consider the lightlike geodesics with trajectories orthogonal to $\d_t$ instead.
Our results in this paper generalize the main results of \cite{PV2023}, where the same statements were proven for $\T$, when~${r_0 \in (r_{\rm e}, r_{\rm c})}$ was the unique point such that $\mu'(r_0)= 0$.

Besides using the computations in \cite{PV2023}, this paper relies heavily on the microlocal analysis developed in \cite{V2013}, in particular on the radial point estimates and the Fredholm theory for non-elliptic operators.
For our application to wave equations, we rely on microlocal estimates near normally hyperbolic trapping in the sense of Wunsch and Zworski in \cite{WZ2011}, see also the improved results by Dyatlov in \cite{Dya2015, Dya2015b, Dya2016}.
For more references on results related to this paper, we refer to the introductions in \cite{PV2023,PV2021}.

\subsection{Kerr--de~Sitter spacetimes}
The geometry of the Kerr--de~Sitter spacetimes depends on a certain polynomial, given by
\begin{equation} \label{eq: mu}
	\mu(r)
		:= - \frac{\Lambda r^4}3 + \left( 1 - \frac{\Lambda a^3}3 \right) r^2 - 2mr + a^2.
\end{equation}

\begin{Definition} \label{def: BL metric}
Assume that $\mu$ has four distinct real roots $r_- < r_{\rm C} < r_{\rm e} < r_{\rm c}$. The manifold
\[
	M := \R_t \times (r_{\rm e}, r_{\rm c})_r \times S^2_{\phi, \theta},
\]
with real analytic metric
\begin{align*}
	g
		={}& \big(r^2 + a^2 \cos^2(\theta)\big)\left( \frac{\md r^2}{\mu(r)} + \frac{\md \theta^2}{c(\theta)}\right) + \frac{c(\theta)\sin^2(\theta)}{b^2\big(r^2 + a^2 \cos^2(\theta)\big)}\big(a \md t - \big(r^2 + a^2\big)\md \phi\big)^2 \\
		& - \frac{\mu(r)}{b^2\big(r^2 + a^2 \cos^2(\theta)\big)}\big(\md t - a \sin^2(\theta)\md \phi\big)^2,
\end{align*}
where
\[
	b := 1 + \frac{\Lambda a^2}3, \qquad c(\theta) := 1 + \frac{\Lambda a^2}3 \cos^2(\theta),
\]
is called the domain of outer communication in a \emph{subextremal Kerr--de~Sitter spacetime} (in Boyer--Lindquist coordinates).
\end{Definition}
One easily verifies that this metric extends real analytically to the north and south poles $\theta = 0, \pi$.

\begin{Remark}
Note that $\d_t$ and $\d_\phi$ are Killing vector fields of $g$.
\end{Remark}

The Boyer--Lindquist coordinates used above become singular at the roots of $\mu(r)$.
As a~physical model of a~rotating black hole in an expanding spacetime, however, the two largest roots of~$\mu(r)$ are supposed to point out the position of the event and cosmological horizons.
We extend the coordinates over the future event horizon and future cosmological horizon with the following coordinate change: $t_*:= t - \Phi(r)$, $\phi_*:= \phi - \Psi(r)$,
where $\Phi$ and $\Psi$ satisfy
\begin{align*}
		\Phi'(r)
			= b \frac{r^2 + a^2}{\mu(r)} f(r), \qquad
		\Psi'(r)
			= b \frac a{\mu(r)} f(r)
\end{align*}
and
\[
	f\colon\ (r_{\rm e} - \de, r_{\rm c} + \de) \to \R
\]
is a real analytic function, for a small enough $\de > 0$ such that $f(r_{\rm e})
		= -1$, $f(r_{\rm c}) = 1$.
The new form of the metric is
\begin{align*}
	g_*
		={}& \big(r^2 + a^2 \cos^2(\theta)\big)\frac{1 - f(r)^2}{\mu(r)} \md r^2 - \frac 2 b f(r) \big(\md t_* - a \sin^2(\theta) \md \phi_*\big)\md r \\
		& - \frac{\mu(r)}{b^2\big(r^2 + a^2 \cos^2(\theta)\big)}\big(\md t_* - a \sin^2(\theta) \md \phi_* \big)^2 \\
		& + \frac{c(\theta)\sin^2(\theta)}{b^2\big(r^2 + a^2 \cos^2(\theta)\big)}\big(a\md t_* - \big(r^2 + a^2\big) \md \phi_*\big)^2 + \big(r^2 + a^2 \cos^2(\theta)\big)\frac{\md \theta^2}{c(\theta)},
\end{align*}
which extends real analytically to
\[
	M_*
		:= \R_{t_*} \times (r_{\rm e} - \de, r_{\rm c} + \de)_r \times S^2_{\phi_*, \theta}.
\]
Now the coordinates are not anymore singular at $r_{\rm e}$ and $r_{\rm c}$ and we get two new real analytic lightlike hypersurfaces
\begin{align*}
	\Ho_e^+
		:= \R_{t_*} \times \{r_{\rm e}\} \times S^2_{\phi_*, \theta}, \qquad
	\Ho_c^+
		:= \R_{t_*} \times \{r_{\rm c}\} \times S^2_{\phi_*, \theta},
\end{align*}
which are called the \emph{future event horizon} and the \emph{future cosmological horizon}, respectively.

\begin{Remark}
The Killing vector fields $\d_t$ and $\d_\phi$ extend to Killing vector fields $\d_{t_*}$ and $\d_{\phi_*}$ over the horizons.
\end{Remark}

\subsection{The first main result}

The main novel observation in this paper is that there are no trapped lightlike geodesics in the domain of outer communication of a subextremal Kerr--de~Sitter spacetime, with trajectories orthogonal to a certain Killing vector field.

\begin{Theorem}[no orthogonal trapping] \label{thm: T-orth trapping}
Let $r_0 \in [r_{\rm e}, r_{\rm c}]$.
All lightlike geodesics in the domain of outer communication $(M,g)$ of a subextremal Kerr--de~Sitter spacetime, with trajectories orthogonal to the Killing vector field
\[
	\T
		:= \d_t + \frac{a}{r_0^2 + a^2} \d_\phi,
\]
eventually leave the region
\[
	\R_{t_*} \times [r_{\rm e}+\e, r_{\rm c}-\e]_r \times S^2
\]
for any $\e > 0$.
Moreover, there is an open subset $\U \subset (r_{\rm e}, r_{\rm c})$, with $r_0 \in \overline{\U}$ and such that no such lightlike geodesic intersects
\begin{equation} \label{eq: timelike region}
	\R_{t_*} \times \U \times S^2.
\end{equation}
\end{Theorem}

\begin{Remark}
A special case of Theorem \ref{thm: asymptotic expansion}, namely when $r_0$ was the unique $r_0 \in (r_{\rm e}, r_{\rm c})$ such that $\mu'(r_0) = 0$,
was proven in \cite[Lemma~2.4]{PV2023} and was one of the two key observations in that paper.
\end{Remark}

\begin{Remark}[the second assertion in Theorem \ref{thm: T-orth trapping}] \label{rmk: second assertion}
The second assertion in Theorem \ref{thm: T-orth trapping} is an immediate consequence of the fact that $\T$ is timelike in a region of the form \eqref{eq: timelike region}.
To check this, we compute
\[ 
	g_*(\T,\T)|_{r = r_0}
		= - \frac{\mu(r_0)\big(r_0^2 + a^2 \cos^2(\theta)\big)}{b^2\big(r_0^2 + a^2\big)^2},
\]
which is negative if $r_0 \in (r_{\rm e}, r_{\rm c})$ and vanishes if $r_0 = r_{\rm e}$ or $r_{\rm c}$.
If $r_0 \in (r_{\rm e}, r_{\rm c})$, it follows that $\T$ is \emph{timelike} at
\begin{equation} \label{eq: causal vector}
	\R_{t_*} \times \{r_0\} \times S^2,
\end{equation}
and it is therefore timelike in an open neighborhood of \eqref{eq: causal vector}.
In this region, there can therefore be no lighlike geodesics with trajectories orthogonal to $\T$, proving the second assertion in Theorem~\ref{thm: T-orth trapping} in this special case.
If instead $r_0 = r_{\rm e}$ or $r_{\rm c}$, then $\T$ is lighlike at \eqref{eq: causal vector}, and we cannot immediately deduce that $\T$ will be causal in a neighborhood.
We therefore compute
\begin{align*}
	\d_r g_*(\T, \T)|_{r = r_{\rm e/c}}
		= \frac{-\mu'(r_{\rm e/c}) \big(r_{\rm e/c}^2 + a^2\cos^2(\theta)\big)}{b^2 \big(r_{\rm e/c}^2 + a^2\big)^2}.
\end{align*}
If now $r_0 = r_{\rm e}$, then the right hand side is negative, which implies that $\T$ is timelike in a subset of the form
\[
	\R_{t_*} \times (r_{\rm e}, r_{\rm e} + \gamma) \times S^2,
\]
for some $\gamma > 0$.
Similarly, if $r_0 = r_{\rm c}$, then $\T$ is timelike in a subset of the form
\[
	\R_{t_*} \times (r_{\rm c} - \gamma, r_{\rm c}) \times S^2,
\]
for some $\gamma > 0$.
This completes the proof of the second assertion in Theorem \ref{thm: T-orth trapping}.
\end{Remark}

\begin{Remark}[the new ergoregions] \label{rmk: time inv}
The computations in the previous remark raise the question whether we can choose $r_0 \in [r_{\rm e}, r_{\rm c}]$ such that $\T$ is \emph{timelike} everywhere in the domain of outer communication in the Kerr--de~Sitter spacetime, and lightlike at the horizons.
Let us compute the Lorentzian length of $\T$ at the horizons:
\begin{align*}
	g(\T,\T)|_{r = r_{\rm e}}
		= \frac{a^2 c(\theta)\sin^2(\theta)}{b^2\big(r^2 + a^2 \cos^2(\theta)\big)}\left(\frac{r_0^2-r_{\rm e}^2}{r_0^2 + a^2}\right)^2, \\
	g(\T,\T)|_{r = r_{\rm c}}
		= \frac{a^2 c(\theta)\sin^2(\theta)}{b^2\big(r^2 + a^2 \cos^2(\theta)\big)}\left(\frac{r_0^2-r_{\rm c}^2}{r_0^2 + a^2}\right)^2.
\end{align*}
In the Schwarzschild--de Sitter spacetime, when $a = 0$, both expressions actually vanish at the horizons and one can easily check that $\T$ is timelike at any $r \in (r_{\rm e}, r_{\rm c})$.
However, if $a \neq 0$, then these expressions only vanish if $r_0 = r_{\rm e}$ or $r_0 = r_{\rm c}$, respectively.
If $r_0 \in (r_{\rm e}, r_{\rm c})$, then both values are positive.
This shows that $\T$ is spacelike at least at one of the horizons.
By analogy with the classical terminology, we call the regions in the domain of outer communication, where $\T$ is spacelike, the \emph{ergoregions} with respect to $\T$.
These computations show that there are two ergoregions if $r_0 \in (r_{\rm e}, r_{\rm c})$, which are non-intersecting by the results in Remark \ref{rmk: second assertion}, and only one ergoregion if $r_0 = r_{\rm e}$ or $r_{\rm c}$.
As a comparison, with the classical choice $\d_t$ as the stationary Killing vector field, the two ergoregions intersect for large $a$ which is undesirable for the analysis.
\end{Remark}

\subsection{The second and third main results}\label{subsec:second result}
We begin with our assumptions:

\begin{Assumption} \label{ass: main wave} \
\begin{itemize}\itemsep=0pt
	\item Let $(M_*, g_*)$ be a subextremal Kerr--de~Sitter spacetime, extended over the future event horizon and the future cosmological horizon, where
\[
	M_* := \R_{t_*} \times (r_{\rm e} - \de, r_{\rm c} + \de)_r \times S^2_{\phi_*, \theta},
\]
with $\de > 0$ small enough so that the boundary hypersurfaces
\[
	\R_{t_*} \times \{r_{\rm e} - \de\} \times S^2_{\phi_*, \theta}, \qquad \R_{t_*} \times \{r_{\rm c} + \de\} \times S^2_{\phi_*, \theta}
\]
are \emph{spacelike} and with $f$ chosen as in {\rm \cite[Remark~1.1]{PV2023}} so that the hypersurfaces
\[
	\{t_* = c \} \times (r_{\rm e} - \de, r_{\rm c} + \de)_r \times S^2_{\phi_*, \theta}
\]
are \emph{spacelike}, for all $c \in \R$.
\item Let $A$ be a smooth complex function on $M_*$ such that
$
	\d_{t_*} A = \d_{\phi_*} A = 0$.
We let $P$ be the linear wave operator given by
$
	P = \Box + A$,
where $\Box$ denotes the d'Alembert operator on scalar functions on $M_*$.
\end{itemize}
\end{Assumption}

For any subset $\U \subset M_*$, we use the notation $C^\infty(\U)$ for the smooth complex functions on $\U$.
As in Theorem \ref{thm: T-orth trapping}, we let
\[
	\T
		:= \d_{t_*} + \frac{a}{r_0^2 + a^2} \d_{\phi_*},
\]
for any fixed $r_0 \in [r_{\rm e}, r_{\rm c}]$.

\subsubsection{Quasinormal modes}

One of the novelties in \cite{PV2023} was a new definition of quasinormal modes with respect to the Killing vector field $\T$, with $r_0 \in (r_{\rm e}, r_{\rm c})$ uniquely determined by the condition $\mu'(r_0) = 0$, instead of the standard choice of Killing vector field $\d_{t_*}$.
In this paper, we show that the analogous result holds as in \cite{PV2023}, if we choose to define our quasinormal modes with respect to $\T$, for any choice of $r_0 \in [r_{\rm e}, r_{\rm c}]$.

\begin{Definition}[quasinormal mode] \label{def: QNMs}
Let $r_0 \in [r_{\rm e}, r_{\rm c}]$.
A complex function $u \in C^\infty(M_*)$ is called a {\em quasinormal mode}, with {\em quasinormal mode frequency} $\s \in \C$, if $\T u = -{\rm i} \s u$
and $P u = 0$.
\end{Definition}

\begin{Remark}
Quasinormal modes and mode frequencies are also called resonant states and resonances.
\end{Remark}
\begin{Remark}
Note that we can write any quasinormal mode as
$
	u = {\rm e}^{-{\rm i} \s t_*}v_\s$,
where $\T v_\s = 0$.
\end{Remark}

Our second main result is the following:

\begin{Theorem}[discrete set of quasinormal modes] \label{thm: QNMs}
Let $(M_*, g_*)$ and $P$ be as in Assumption~{\rm\ref{ass: main wave}}.
Then there is a discrete set $\A \subset \C$ such that $\s \in \A$
if and only if there is a quasinormal mode~${u \in C^\infty(M_*)}$
with mode frequency $\s$.
Moreover, for each $\s \in \A$, the space of quasinormal modes is finite dimensional.
If the coefficients of $P$ are real analytic, then the quasinormal modes are real analytic.
\end{Theorem}

\noindent
This result generalizes \cite[Theorem~1.5]{PV2023} by allowing us to choose any $r_0 \in [r_{\rm e}, r_{\rm c}]$ in the definition of quasinormal modes, instead of only the unique one such that $\mu'(r_0) = 0$.
For more comments on the statement of Theorem~\ref{thm: QNMs}, we therefore refer to the discussion following \cite[Theorem~1.5]{PV2023}.

\subsubsection{Asymptotic expansion}

The last main result concerns the asymptotics of solutions to linear wave equations when $t_* \to \infty$.
We formulate the statement using the standard Sobolev spaces on
\[
	M_*
		= \R_{t_*} \times (r_{\rm e} - \de, r_{\rm c} + \de)_r \times S^2_{\phi_*, \theta},
\]
i.e., the ones associated with the Riemannian metric $\md t_*^2 + \md r^2 + g_{S^2}$,
where $g_{S^2}$ is the round metric on the 2-sphere.
For non-negative integers $s$, a Sobolev norm (unique up to equivalence) is given by
\[
	\|u\|_{\bar H^s}^2=\sum_{i + j + k \leq s} \bigl\|\partial_{t_*}^i \partial_r^j \left( \Delta_{S^2}+1\right)^{k/2} u\bigr\|^2_{L^2(M_*)}.
\]
The bar over $H$ corresponds to H\"ormander's notation for
extendible distributions, see \cite{HorIII}.
We have the following statement.

\begin{Theorem}[the asymptotic expansion of waves] \label{thm: asymptotic expansion} Let $(M_*, g_*)$ and $P$ be as in Assumption~{\rm\ref{ass: main wave}} and let $t_0 \in \R$.
There are $C, \de > 0$ such that for $0 < \e < C$ and $s > \frac12 + \b \e$, where
\[
	\b
		:= 2b \max_{r \in \{r_{\rm e}, r_{\rm c}\}} \left( \frac{r^2 + a^2}{\abs{\mu'(r)}} \right),
\]
any solution to $Pu = f$
with $f \in {\rm e}^{- \e t_*} \bar H^{s - 1 + \de}(M_*)$ and with $\supp(u) \cup \supp(f) \subset \{t_* > t_0\}$ has an asymptotic expansion
\begin{align*}
	u - \sum_{j = 1}^N \sum_{k = 0}^{k_j} t_*^k {\rm e}^{-{\rm i} \s_j t_*} v_{jk} \in {\rm e}^{- \e t_*} \bar H^s(M_*),
\end{align*}
where $\s_1, \hdots, \s_N$ are the $($finitely many$)$ quasinormal mode frequencies with $\Im \s_j > - \e$
and~$k_j$ is their multiplicity, and where ${\rm e}^{-{\rm i}\s_j t_*}v_{jk}$ are
the $C^\infty$ $($generalized$)$ quasinormal modes with frequency $\s_j$ which are
real analytic if the coefficients of $P$ are such.
\end{Theorem}

Analogously to above, this result generalizes \cite[Theorem~1.6]{PV2023} by allowing us to choose any~${r_0 \in [r_{\rm e}, r_{\rm c}]}$ in the definition of quasinormal modes, instead of only the unique one such that $\mu'(r_0) = 0$.
For more comments on the statement of Theorem~\ref{thm: asymptotic expansion}, we therefore refer to the discussion following \cite[Theorem~1.6]{PV2023}.

Theorem \ref{thm: asymptotic expansion} is naturally combined with mode stability results, i.e.~statements saying that under suitable assumptions there cannot be any modes with $\Im(\s) \geq 0$,
except certain geometric modes where $\s = 0$.
Indeed, in view of Theorem~\ref{thm: asymptotic expansion}, mode stability would imply exponential decay to the zero mode.
One such statement, for the standard d'Alembert wave operator $\Box$, was recently proven by Hintz in \cite{H2021}.
Combining \cite[Theorem~1.1]{H2021} with Theorem~\ref{thm: asymptotic expansion} gives the following decay statement for a certain range of Kerr--de~Sitter parameters $\Lambda$, $a$ and $m$:

\begin{Corollary}
Let $(M_*, g_*)$ be as in Assumption {\rm\ref{ass: main wave}} and let $t_0 \in \R$.
Assume that $\frac {\abs a}m < 1$.
Then there is a $\gamma > 0$ such that if $\Lambda m^2 < \gamma$, then there are $C, \de > 0$ such that for $0 < \e < C$ and $s > \frac12 + \b \e$,
where
\[
	\b
		:= 2b \max_{r \in \{r_{\rm e}, r_{\rm c}\}} \left( \frac{r^2 + a^2}{\abs{\mu'(r)}} \right),
\]
any solution to $\Box u = f$
with $f \in {\rm e}^{- \e t_*} \bar H^{s - 1 + \de}(M_*)$ and with $\supp(u) \cup \supp(f) \subset \{t_* > t_0\}$ satisfies
\begin{align*}
	u - c \in {\rm e}^{- \e t_*} \bar H^s(M_*),
\end{align*}
for some constant $c \in \C$.
\end{Corollary}

\section{The T-orthogonal lightlike geodesics}

The goal of this section is to prove Theorem \ref{thm: T-orth trapping}.
For this, let $r_0 \in [r_{\rm e}, r_{\rm c}]$
and
\[
	\T
		= \d_t + \frac a{r_0^2 + a^2} \d_\phi
\]
throughout the section.
When computing properties of lightlike geodesics, we are going to use the Hamiltonian formalism.
The Hamiltonian for geodesics is the metric $G$ dual to $g$, considered as a function on the cotangent bundle of $M$:
\begin{align*}
	G\colon\ T^*M
		\to T^*M, \qquad
	\xi
		\mapsto G(\xi, \xi).
\end{align*}
The dual metric $G$ of $g$ is given in Boyer--Lindquist coordinates by
\begin{gather}
	\big(r^2 + a^2 \cos^2(\theta)\big)G(\xi, \xi)\nonumber\\
		\qquad= \mu(r) \xi_r^2 + \frac{b^2}{c(\theta)\sin^2(\theta)}\big(a \sin^2(\theta)\xi_t + \xi_\phi\big)^2 - \frac{b^2}{\mu(r)} \big(\big(r^2 + a^2\big) \xi_t + a \xi_\phi \big)^2 + c(\theta) \xi_\theta^2,
 \label{eq: full Hamiltonian}
\end{gather}
where $\xi = (\xi_t, \xi_r, \xi_\phi, \xi_\theta)$ are the dual coordinates to $(t, r, \phi, \theta)$.
Since the bicharacteristic flow is invariant under conformal changes, we study from now on the Hamiltonian
\[
	\q(\xi)
		:= \big(r^2 + a^2 \cos^2(\theta)\big)G(\xi, \xi),
\]
given in \eqref{eq: full Hamiltonian}.
Since $\d_t$ and $\d_\phi$ are Killing vector fields, the dual variables $\xi_t$ and $\xi_\phi$ will be constant along the Hamiltonian flow.
Note that a vector $v \in TM$ is orthogonal to $\T$ if and only if the metric dual vector $\xi
		:= g(v, \cdot)$ satisfies
\begin{equation} \label{eq: F vanishing}
	0
		= g(v, T)
		= \xi(T)
		= \xi_t + \frac a{r_0^2 + a^2} \xi_\phi.
\end{equation}
We may now prove Theorem \ref{thm: T-orth trapping}.

\begin{proof}[Proof of Theorem \ref{thm: T-orth trapping}]
The main step in the proof is to show a convexity property for the radial function $r$ along the bicharacteristic flow in the domain of outer communication.
More precisely, we want to prove that
\begin{equation}\label{eq: convexity in r}
	\H_\q r
		= 0 \quad \Rightarrow \quad \sgn\big( \H_\q^2 r \big) = \sgn(r - r_0)
\end{equation}
at all points in the characteristic set in the domain of outer communication.
Here, the Hamiltonian vector field is given by
\[
	\H_{\q}
		:= \sum_{j = 1}^4 (\d_{\xi_j}\q)\d_j - (\d_j \q)\d_{\xi_j}.
\]
We compute $\H_{\q} r
		= 2 \mu(r) \xi_r$.
Assuming that $\H_{\q} r = 0$ at some $r \in (r_{\rm e}, r_{\rm c})$, we conclude that~${\xi_r = 0}$ there.
The second derivative along the Hamiltonian flow at such a point is given by
\begin{align}
	\H_{\q}^2 r|_{\xi_r = 0}
		&= 2 \mu(r) \H_{\q} \xi_r|_{\xi_r = 0} = - 2 \mu(r) \d_r \left( - \frac{b^2}{\mu(r)} \big( \big(r^2 + a^2\big) \xi_t + a \xi_\phi \big)^2 \right) \nonumber\\
		&= 2 \mu(r) b^2 \d_r \frac{\big( \big(r^2 + a^2\big) \xi_t + a \xi_\phi \big)^2}{\mu(r)}.\label{eq: second derivative Hamiltonian}
\end{align}
Recall that $\mu(r) > 0$ at all points in the domain of outer communication.
By \cite[Theorem~3.2\,(a)]{PV2023}, the function
\[
	F(r)
		:= \frac{\big( \big(r^2 + a^2\big) \xi_t + a \xi_\phi \big)^2}{\mu(r)}
\]
either vanishes at $r_{\rm e}$ or $r_{\rm c}$ and has no critical point in $(r_{\rm e}, r_{\rm c})$ or has precisely one critical point in $(r_{\rm e}, r_{\rm c})$.
Since $F$ is a non-negative function and vanishes at $r_0$ by \eqref{eq: F vanishing}, $F$ can have no other critical points than $r_0$.
It follows that $F'(r)$, and therefore the right hand side of \eqref{eq: second derivative Hamiltonian}, has the signs claimed in \eqref{eq: convexity in r}.

We now define an escape function $\E := {\rm e}^{C(r - r_0)^2}\H_{\q_\s} r$
for any $C > 0$ and note that
\[
	\H_{\q_\s} \E
		= {\rm e}^{C(r - r_0)^2} \big( 2 C(r - r_0) (\H_{\q_{\s}}r)^2 + \H_{\q_\s}^2 r \big).
\]
Since the characteristic set is disjoint from $\{r = r_0\}$, by Remark \ref{rmk: second assertion}, and since $\H_\q^2 r$ has the same sign as $r - r_0$ whenever $\H_\q r$ vanishes, by the implication \eqref{eq: convexity in r}, and since the Hamiltonian flow is invariant under translations in $\R_{t_*}$, we can choose the constant $C$ large enough to make sure that $\H_{\q_\s} \E$ is nowhere vanishing on
\[
	\R_{t_*} \times [r_{\rm e} + \e, r_{\rm c} - \e] \times S^2
\]
and has the same sign as $r - r_0$.
Hence $\E$ gives an escape function for all bicharacteristics satisfying \eqref{eq: F vanishing}.
This finishes the proof of Theorem \ref{thm: T-orth trapping}.
\end{proof}

\section{Fredholm theory}

The purpose of this section is to prove Theorems \ref{thm: QNMs} and \ref{thm: asymptotic expansion}.
For this, let again $r_0 \in [r_{\rm e}, r_{\rm c}]$
and
\[
	\T
		= \d_{t_*} + \frac a{r_0^2 + a^2} \d_{\phi_*}
\]
throughout the section.
The theory of wave equations in \cite{V2013} is based on first Fourier transforming the wave operator in the variable $t_*$.
We thus want to consider the induced operator
\[
	P_\s v
		:= {\rm e}^{{\rm i}\s t_*} P({\rm e}^{-{\rm i} \s t_*} v),
\]
for a fixed $\s \in \C$, where
$\T v
		= 0$.
This latter condition can be interpreted as $v$ being only dependent on a certain set of coordinates, cf.\ \cite[equation~(11)]{PV2023}, but this viewpoint is not necessary for the discussion here.
Since $\T$ is a Killing vector field, and the coefficients of $P$ are invariant under $\T$ (cf.~Assumption \ref{ass: main wave}), we can think of $P_\s$ as a differential operator
\[
	P_\s
		\colon\
 C^\infty(L_*) \to C^\infty(L_*),
\]
where
\[
	L_*
		:= t_*^{-1}(0) \subset M_*
\]
is a spacelike hypersurface.
Now, since $\s \in \C$ is fixed, the bicharacteristic flow of $P_\s$ is canonically identified with the lightlike geodesics in $M_*$ with trajectories orthogonal to $\T$.
Thus, we already know by Theorem \ref{thm: T-orth trapping} that the bicharacteristic flow of $P_\s$ is non-trapping in the domain of outer communication.
Following \cite{V2013}, we want to show that $P_\s$ is in fact a Fredholm operator between appropriate Sobolev spaces.
For any $s \in \R$, we write
$ \bar H^s := \bar H^s(L_*)$,
for the space of extendible Sobolev distributions on $L_*$, in the sense of H\"ormander \cite{HorIII}, of degree $s$.
The following is an improvement over \cite[Lemma~2.1]{PV2023} and \cite[Theorem~1.1]{V2013}, where we allow to choose any $r_0 \in [r_{\rm e}, r_{\rm c}]$ in the definition of $\T$.

\begin{Theorem} \label{thm: Fredholm}
Define
\[
	\b
		:= 2b \max_{r \in \{r_{\rm e}, r_{\rm c}\}} \left( \frac{r^2 + a^2}{\abs{\mu'(r)}} \right),
\]
and let $s \geq \frac12$.
The operator
\[
  P_\s\colon\ \big\{u \in \bar H^s \mid P_\s u \in \bar H^{s-1} \big\} \to \bar H^{s-1}
\]
is an analytic family of Fredholm operator of index $0$ for all $\s \in \C$ such that $\Im \s > \frac{1 - 2s}{2\b}$.
Moreover, $P_\s$ is invertible for $\Im(\s) \gg 1$.
\end{Theorem}

The proof of Theorem \ref{thm: Fredholm} relies on the Fredholm theory for non-elliptic operators developed in \cite{V2013}, which requires refined understanding of the behavior of the bicharacteristics.
The number~${\b > 0}$ in Theorem \ref{thm: Fredholm} is related to the surface gravity of the horizons and corresponds to a~threshold in the radial point estimates in \cite{V2013}.
Note that there is a canonical identification
\[
	\Char(P_\s)
		\subset \Char(P)
		\subset T^*M_*.
\]
We may therefore define
\[
	\S_\pm
		= \{ \xi \in \Char(P_\s) \mid \pm G_*(\md t_*, \xi) > 0 \},
\]
and note that $\S_- \cap \S_+ = \varnothing$, which in particular implies that $\S_-$ and $\S_+$ are invariant under the bicharacteristic flow.
Moreover, since $L_*$ is spacelike by Assumption \ref{ass: main wave}, it follows that $\md t_*$ is timelike along $L_*$ and hence
\[
	\Char(P_\s)
		= \S_- \cup \S_+.
\]
Analogously to \cite[Lemma~2.5]{PV2023} and \cite[Section~6]{V2013}, we have the following.
\begin{Proposition} \label{prop: classical radial points}
Let $\xi_r
		:= \xi(\d_r)$,
for any $\xi \in \Char(P_\s)$.
The conormal bundles $N^*\{r = r_{\rm e}\}$ and $N^*\{r = r_{\rm c}\}$ are contained in the characteristic set of $P_\s$ and the bicharacteristic flow is radial in the generalized sense as in {\rm\cite{V2013}} at these.
All other bicharacteristics of $P_\s$ in $\S_+$ either start at fiber infinity of
\[
	N^*\{r = r_{\rm e}\} \cap \{\xi_r > 0\}
\]
and end at $r = r_{\rm e} - \de$ or start at the fiber infinity of
\[
	N^*\{r = r_{\rm c}\} \cap \{\xi_r < 0\}
\]
and end at $r = r_{\rm c} + \de$.
All other bicharacteristics of $P_\s$ in $\S_-$ either start at $r = r_{\rm e} - \de$ and end at fiber infinity of
\[
	N^*\{r = r_{\rm e}\} \cap \{\xi_r < 0\}
\]
or start at $r = r_{\rm c} + \de$ and end at the fiber infinity of
\[
	N^*\{r = r_{\rm c}\} \cap \{\xi_r > 0\}.
\]
Moreover, the fiber infinity of
\[
	N^*\{r = r_{\rm e}\} \cap \{\pm \xi_r > 0\} \text{ and } N^*\{r = r_{\rm c}\} \cap \{\mp \xi_r > 0\}
\]
are generalized normal source/sink manifolds of the bicharacteristic flow, respectively, in the sense of {\rm\cite{V2013}}.
Furthermore, if $r_0 = r_{\rm e}$, then no bicharacteristics between the fiber infinities of~${N^*\{r = r_{\rm e}\}}$ and $\{r = r_{\rm e} - \de\}$ intersect the domain of outer communication $M$ and
\[N^*\{r = r_{\rm e}\} \cap \{\pm \xi_r > 0\}\]
 is a stable radial point source/sink, in the sense of {\rm\cite{V2013}}.
If $r_0 = r_{\rm c}$, then the corresponding holds at the cosmological horizon.
\end{Proposition}

\begin{proof}
As explained above, Theorem \ref{thm: T-orth trapping} implies that no bicharacteristics of $P_\s$ are trapped in~${L_* \cap M}$.
Moreover, it implies that no bicharacteristics of $P_\s$ can approach the event horizon to the past/future and the cosmological horizon to the future/past.

Following the argument in the proof of \cite[Lemma~2.5]{PV2023}, the first two statements in the proposition would follow by combining Theorem \ref{thm: T-orth trapping} with the semiclassical considerations in \cite{V2013} near the horizons.
However, just like in the proof of \cite[Lemma~2.5]{PV2023}, we choose here to present the details for the convenience of the reader.
In the proof of \cite[Lemma~2.5]{PV2023}, the special case when $r_0 \in (r_{\rm e}, r_{\rm c})$ with $\mu'(r_0) = 0$ was considered.
The exact same argument as written there goes through, line by line, in the case when $r_0 \in (r_{\rm e}, r_{\rm c})$, using Theorem \ref{thm: T-orth trapping}.
Only the cases when $r_0 = r_{\rm e}$ or $r_0 = r_{\rm c}$ require some extra care.
Let us only discuss the case when $r_0 = r_{\rm e}$, since the other case is similar.
For the bicharacteristics starting or ending at $r = r_{\rm c} + \de$, the same analysis as in the proof of \cite[Lemma~2.5]{PV2023} applies.
The bicharacteristic flow at the event horizon is, however, slightly different.
Though a similar computation was already done in the proof of \cite[Theorem~1.7]{PV2021}, let us explicitly compute the Hamiltonian vector field at $N^*\{r = r_{\rm e}\}$ in case $r_0 = r_{\rm e}$.
The principal symbol $\p_\s$ of $P_\s$ is given by
\begin{align*}
	&\big(r^2 + a^2\cos^2(\theta)\big)\p_\s(\xi) \\
		&\qquad= \mu(r) \xi_r^2 - 2 a b f(r) \frac{r_{\rm e}^2 - r^2}{r_{\rm e}^2 + a^2} \xi_{\phi_*} \xi_r + c(\theta) \xi_\theta^2 \\
		&\qquad\phantom{=}{} + \left( \frac{b^2}{c(\theta)\sin^2(\theta)}\left(\frac{r_{\rm e}^2 + a^2 \cos^2(\theta)}{r_{\rm e}^2 + a^2} \right)^2 - b^2 \frac{1 - f(r)^2}{\mu(r)} \left(a \frac{r_{\rm e}^2 - r^2}{r_{\rm e}^2 + a^2} \right)^2 \right) \xi_{\phi_*}^2.
\end{align*}
It immediately follows that $N^*\{r = r_{\rm e}\} \subset \Char(P_\s)$ and the Hamiltonian vector field at ${N^*\{r = r_{\rm e}\}}$ is given by
\[
	\H_{\p_\s}
		= \big(r^2 + a^2\cos^2(\theta)\big)^{-1} \mu'(r_{\rm e}) \xi_r^2 \d_{\xi_r}.
\]
It follows that the bicharacteristic flow at the conormal bundle of $\{r = r_{\rm e}\}$ is exactly \emph{radial}, as opposed to radial in the generalized sense.
The stability of the source/sink can be show, for example, as in the proof of \cite[Lemma~2.5]{PV2023}.
The fact that no bicharacteristics between~${\{r = r_{\rm e} - \de\}}$ and the fiber infinities of $N^*\{r = r_{\rm e}\}$ intersect the domain of outer communication~$M$ is an immediate consequence of Remark \ref{rmk: second assertion}.
The analogous computations at $r = r_{\rm c}$ when $r_0 = r_{\rm c}$ complete the proof.
\end{proof}

\begin{proof}[Proof of Theorem \ref{thm: Fredholm}]
By Proposition \ref{prop: classical radial points}, the dynamics of the bicharacteristics of $P_\s$ in $L_*$ precisely analogous as in \cite[Section~6.1]{V2013}.
The proof of Theorem \ref{thm: Fredholm} therefore follows the same lines as the proof of \cite[Theorem~1.4]{V2013}.
\end{proof}

\begin{proof}[Proof of Theorem \ref{thm: QNMs}]
We again consider the analytic Fredholm family
\[
  P_\s\colon\ \big\{u \in \bar H^s \mid P_\s u \in \bar H^{s-1} \big\} \to \bar H^{s-1}
\]
from Theorem \ref{thm: Fredholm}.
A standard energy estimate shows that $P_\s$ is invertible for $\Im(\s) \gg 1$.
Analytic Fredholm theory therefore implies that $P_\s$ has a meromorphic extension to the open set
\[
	\Omega_s
		:= \left\{ \Im (\s) > \frac{1 - 2s}{2\b} \right\}.
\]
In particular, $P_\s$ is invertible everywhere in $\Omega_s$ apart form a discrete set.
Moreover, since $P_\s$ has index zero, $P_\s$ is invertible if and only if the kernel of $P_\s$ is trivial.
Since
\[
	\C = \bigcup_{s \in \R} \Omega_s,
\]
we conclude that $\ker(P_\s)$ is non-trivial precisely on a discrete set $\A \subset \C$.
Following the arguments in the proof of \cite[Theorem~1.2]{PV2021} line by line, using Theorem \ref{thm: Fredholm} in place of the \cite[Theorem~1.1]{V2013}, it follows that the elements in $\ker(P_\s)$ are real analytic if the coefficients of $P$ are real analytic.
\end{proof}

\begin{proof}[Proof of Theorem \ref{thm: asymptotic expansion}]
We again consider the analytic Fredholm family
\[
  P_\s\colon\ \big\{u \in \bar H^s \mid P_\s u \in \bar H^{s-1} \big\} \to \bar H^{s-1}
\]
from Theorem \ref{thm: Fredholm}.
The semi-classical trapping is corresponding to the trapping of bicharacteristics of the full wave operator $P$.
Since \cite[Theorem~3.2]{PV2023} implies that the trapping of bicharacteristics of $P$ is normally hyperbolic, the proof of the semi-classical estimates and consequently the proof of Theorem \ref{thm: asymptotic expansion} proceeds completely analogous to the proof of \cite[Theorem~1.4]{V2013}.
\end{proof}

\subsection*{Acknowledgements}
This paper is dedicated to Christian B\"ar's 60th birthday.
We would like to thank him for all his inspiring work.
The first author gratefully acknowledges the support from Swedish Research Council under grant number 2021-04269.
The second author gratefully acknowledges support from the National Science Foundation under grant number DMS-1953987.

\pdfbookmark[1]{References}{ref}
\LastPageEnding

\end{document}